\documentclass[11pt, reqno]{amsart}

\usepackage{amssymb,latexsym,amsmath,amsfonts,amsthm}
\usepackage{mathrsfs}
\usepackage{enumitem}
\usepackage[usenames]{color}
\usepackage{hyperref}
\usepackage{comment}
\usepackage{mathtools}

\allowdisplaybreaks

\numberwithin{equation}{section}

\definecolor{DPurple}{rgb}{0.46,0.2,0.69}

\theoremstyle{definition}
\newtheorem{definition}{Definition}[section]

\theoremstyle{remark}
\newtheorem{remark}[definition]{Remark}

\theoremstyle{plain}
\newtheorem{theorem}[definition]{Theorem}
\newtheorem{result}[definition]{Result}
\newtheorem{lemma}[definition]{Lemma}
\newtheorem{proposition}[definition]{Proposition}

\newcommand{\eps}{\varepsilon}
\newcommand{\zt}{\zeta}

\newcommand{\poin}{\boldsymbol{{\sf p}}}

\newcommand{\bdy}{\partial}
\newcommand{\OM}{\Omega}
\newcommand{\D}{\mathbb{D}}

\newcommand{\smoo}{\mathcal{C}}
\newcommand{\hol}{\mathcal{O}}

\newcommand{\bcdot}{\boldsymbol{\cdot}}
\newcommand{\lrarw}{\longrightarrow}
\newcommand{\Top}{\text{{\sl Top}}}

\newcommand{\N}{\mathbb{N}}

\newcommand{\Cn}{\mathbb{C}^n}

\newcommand{\C}{\mathbb{C}}

\newcommand{\wt}{\widetilde}

\begin{document}

\title[Kobayashi complete domains]{Kobayashi complete domains in complex manifolds}

\author{Rumpa Masanta}
\address{Department of Mathematics, Indian Institute of Science, Bangalore 560012, India}
\email{rumpamasanta@iisc.ac.in}

\begin{abstract}
In this paper, we give sufficient conditions for Cauchy-completeness of Kobayashi hyperbolic domains in complex
manifolds. The first result gives a sufficient condition for completeness for relatively compact domains in several large
classes of manifolds. This follows from our second result, which may
be of independent interest, in a much more general setting. This extends a result of Gaussier to the setting
of manifolds. 
\end{abstract}


\maketitle

\vspace{-4.25mm}
\section{Introduction and statement of results}\label{S:intro}

This paper is devoted to investigating when, given a complex manifold $X$ and a domain $\OM\varsubsetneq X$ ($\OM$ not
necessarily relatively compact) that is Kobayashi hyperbolic, $\OM$ equipped with the Kobayashi distance $K_\OM$ is
Cauchy-complete. In this paper, our focus will
be on complex manifolds $X$ with $\dim (X)=n$ that cannot be realized as
open subsets of $\C^n$. The question of determining whether $(\OM, K_\OM)$ is Cauchy-complete when $\OM\varsubsetneq 
\Cn$ is already a very hard problem. Still, several results are known in the latter case; see, for instance 
\cite[Theorem~3.5]{BedfordFornaess:acopfowpd78}, \cite[Theorem~4.1.7]{kobayashi:hcs98}, \cite[Theorem~1]{Gaussier:tachodic99}.
\smallskip

Deciding whether $(\OM,K_\OM)$ is Cauchy-complete is an important problem. For instance, Cauchy-completeness of
$(\OM, K_\OM)$ would imply that $(\OM,K_\OM)$ is a geodesic space (this follows from the properties of $K_\OM$ and the
Hopf--Rinow Theorem). This would endow $(\OM,K_\OM)$ with many of the properties of geodesically complete Riemannian
manifolds, which are very well-studied. Furthermore, Cauchy-completeness of $(\OM,K_\OM)$ is indispensable for certain
types of results for analytical reasons: see, for instance \cite{Kaup:hkR68, Kwack:gbPt69, KimVerdiani:cndnrnn2dag04}.
\smallskip

Before discussing the set-up that is the focus of this paper, we present a definition that will
be essential to our discussion.

\begin{definition}\label{D:peak_function}
Let $X$ be a complex manifold and let $\OM\varsubsetneq X$ be a domain in $X$. Let $p\in\bdy\OM$. We say that 
\emph{$p$ admits a local holomorphic weak peak function} if there exists a neighbourhood $V_p$ of
$p$ in $X$ and a continuous function $f_p: (V_p\cap\OM)\cup \{p\}\lrarw \C$ such that $f_p$ is holomorphic on 
$V_p\cap\OM$ and such that 
\[
|f_p(z)|<1 \;\;\forall z\in V_p\cap \OM, \quad f_p(p)=1.
\]
We say that \emph{$p$ admits a holomorphic weak peak function} if $(V_p\cap\OM)$ is replaced by $\OM$
in the latter definition.
\end{definition}

Considering any connected complex manifold  $X$, the result that addresses Cauchy-completeness of
$(X, K_X)$ in this generality says that if $X$ admits a complete Hermitian metric whose holomorphic sectional
curvature is bounded above by a negative constant, then $(X, K_X)$ is Cauchy-complete
\cite[Theorem~4.11]{kobayashi:hmhm70}.
Since it is very unclear when a connected complex manifold admits a complete Hermitian metric just with  
everywhere negative holomorphic sectional curvature, another approach to our
goal might be to identify some class of complex
manifolds $X$ for which, given a domain $\OM\varsubsetneq X$, conditions similar to those in the papers
cited above give Cauchy-completeness of $(\OM, K_{\OM})$. A result of this kind by Bharali \emph{et al.}
\cite[Result~3.6]{BharaliEtAl:phmspRs18} states (paraphrasing slightly): 
\emph{Let $X$ be a Stein manifold and $\OM\Subset X$ be a relatively compact domain. If each point
of $\bdy\OM$ admits a holomorphic weak peak function, then $\OM$ equipped with the Carath{\'e}odory
distance $C_{\OM}$ is a proper metric space and hence is Cauchy-complete.} Note that, since
$C_{\OM}\leq K_{\OM}$, we deduce in this case that $(\OM, K_{\OM})$ is Cauchy-complete. The authors
attribute \cite[Result~3.6]{BharaliEtAl:phmspRs18} to Kobayashi, stating that the proof of
\cite[Theorem~4.1.7]{kobayashi:hcs98} applies to it also. This claim isn't quite accurate: see Zwonek
\cite[Remark~3]{Zwonek:ftpspcm20}, who also observes that the above statement \textbf{would} be true
if $\OM$ were \emph{a priori} Carath{\'e}odory hyperbolic. The claim by Bharali \emph{et al.} still
proves to be useful to the main aim of this paper. This paper responds to their claim in two ways: 
\begin{itemize}[leftmargin=20pt]
  \item We prove Cauchy-completeness of $(\OM, K_{\OM})$ under slightly weaker conditions than
  those in \cite[Result~3.6]{BharaliEtAl:phmspRs18} (the latter result itself can be proved
  but we shall not discuss this).
  \item With $X$ and $\OM$ as in Definition~\ref{D:peak_function}, properness of
  $(\OM, C_{\OM})$ is a more restrictive property than Cauchy-completeness of $(\OM, K_{\OM})$. It would
  thus be useful to bypass all considerations of $C_{\OM}$ and focus on the properties of
  $K_{\OM}$. With this approach, we are able to prove Cauchy-completeness of $(\OM, K_{\OM})$ for a
  class of ambient manifolds $X\varsupsetneq \OM$ much wider than the class of Stein manifolds. 
\end{itemize}
These points are made more precise by the following theorem (the words ``slightly weaker conditions'' above refer
to the fact that we only need \textbf{local} holomorphic weak peak functions in the hypothesis of
our theorem):

\begin{theorem}\label{T:weak_hyperbolic}
Let $X$ be any one of the following:
\begin{itemize}
  \item[$(a)$] A Kobayashi hyperbolic complex manifold,
  \item[$(b)$] A Stein manifold,
  \item[$(c)$] A holomorphic fiber bundle $\pi: X\to Y$ such that
  $Y$ is either of $(a)$ or $(b)$ and the fibers are Kobayashi hyperbolic,
  \item[$(d)$] A holomorphic covering space $\pi: X\to Y$ such that $Y$ is either of $(a)$ or $(b)$.
\end{itemize}
Let $\OM\varsubsetneq X$ be a relatively compact domain. If each point of $\bdy\OM$ admits a local holomorphic weak peak function, then $(\OM, K_\OM)$ is Cauchy-complete.
\end{theorem}

A crucial step in the proof of the above result is the second theorem of this paper. To introduce
this theorem, we need two more definitions.

\begin{definition}\label{D:peak_function_infty}
Let $X$ and $\OM$ be as in Definition~\ref{D:peak_function}, and suppose $X$ is non-compact and $\OM$ is not
relatively compact. Let $X^\infty$ denote the one-point compactification of $X$ and $\bdy_\infty\OM$ the
boundary of $\OM$ viewed as a domain in $X^\infty$. We say that \emph{$\infty$ admits a local holomorphic weak
peak function} if there exists a compact $K\subset\OM$ and a function $f_\infty$ holomorphic on
$\OM\setminus K$ such that 
\[ 
|f_\infty (z)| <1\;\; \forall z\in\OM\setminus K,\quad \lim_{z\to\infty} f_\infty(z)=1.
\]
\end{definition}

\begin{definition}\label{D:hyp_imb}
Let $X$ be a connected complex manifold and let $Z$ be a complex submanifold in $X$. We say that \emph{$Z$ is 
hyperbolically imbedded in $X$} if for each $x\in\overline{Z}$ and for each neighbourhood $V_x$ of $x$ in $X$,
there exists a neighbourhood $W_x$ of $x$ in $X$, $\overline{W_x}\subset V_x$, such that
\[K_Z(\overline{W_x}\cap Z, Z\setminus V_x)>0.\] 
\end{definition}

\begin{remark}\label{Rem:def-hyp_imb}
Clearly, $Z$ is hyperbolically imbedded in $X$ if and only if, for $x,y\in\overline{Z}$ and $x\ne y$, there exist neighbourhoods $V_x$ of
$x$ and $V_y$ of $y$ in $X$ such that $K_Z(V_x\cap Z, V_y\cap Z)>0$.
\end{remark}

\begin{theorem}\label{T:k-complete}
Let $X$ be a connected complex manifold and let $\OM\varsubsetneq X$ be a hyperbolically imbedded 
domain. 
\begin{itemize}
\item[$(a)$] If $\OM$ is relatively compact, assume: \medskip

\begin{itemize}[leftmargin=33pt]
\item[$(*)$] each point in $\bdy\OM$ admits a local holomorphic weak peak function.
\end{itemize}
\medskip

\item[$(b)$] If $\OM$ is not relatively compact (in which case $X$ is non-compact), let $X^\infty$
denote the one-point compactification of $X$, assume the condition $(*)$, and assume that $\infty$ 
admits a local holomorphic weak peak function.
\end{itemize}
Then, $(\OM, K_\OM)$ is Cauchy-complete.
\end{theorem}

Note the resemblance of Theorem~\ref{T:k-complete} to \cite[Theorem~1]{Gaussier:tachodic99}
by Gaussier. In fact, Theorem~\ref{T:k-complete} is inspired by Gaussier's result. That said, there are two differences
between the \textbf{hypotheses} of our result and \cite[Theorem~1]{Gaussier:tachodic99} which allows us to
extend Gaussier's result to the setting of manifolds and to give a purely metric-geometry proof. We will elaborate
upon this in Section~\ref{S:k-complete}.

\subsection{A remark on our notation}
A couple of notations will recur throughout this paper that will have certain assumptions associated with them. Let
us clarify these assumptions.
\begin{enumerate}
    \item[$(a)$] Let $X$ be a complex manifold and $G$ be an open set in $X$. We define the function $K_G: G\times G\lrarw [0,\infty]$ by
    \begin{equation}
     K_G(z,z'):=
     \begin{cases}
      K_{\wt{G}}(z,z') &\text{if $z$, $z'$ are in same connected component $\wt{G}$ of $G$,}\notag\\
      \infty &\text{if $z$, $z'$ are in different connected components of $G$},
     \end{cases}
    \end{equation}
    for all $z$, $z'\in G$. Here, $K_{\wt{G}}$ is the Kobayashi pseudodistance on the domain $\wt{G}$ in $X$ (i.e., $K_{\wt{G}}:\wt{G}\times\wt{G}\lrarw [0,\infty)$).
    
    \item[$(b)$] Let $Z$ and $X$ be complex manifolds. Then $\hol(Z,X)$ will denote the space of holomorphic mappings from $Z$ into $X$ equipped with the compact-open topology.
    \end{enumerate}
\smallskip

\section{Preliminary lemmas}\label{S:prelim}

This section is devoted to a number of technical lemmas that will be needed in proving the main theorems.

\begin{lemma}\label{L:hyp}
Let $X$ be a connected complex manifold and $\OM$ be a Kobayashi hyperbolic domain in $X$. Let $V$, $W$ be open 
sets in $X$ such that $W \subseteq V$, $W\cap\OM\neq\emptyset$, and such that
$\delta:=K_\OM(W\cap\OM, \OM\setminus V) >0$. Fix $\eps\in (0,\delta /2)$. Then, 
there exists a constant $C>0$ such that for all 
$z,z'\in W\cap\OM$ satisfying $K_\OM(z, z') <\eps$, we have $K_{V\cap\OM}(z, z') <C\eps$.
\end{lemma}

This lemma is, in essence, \cite[Lemma~3.3.5]{kobayashi:hcs98} by Kobayashi, just written
differently. Firstly, \cite[Lemma~3.3.5]{kobayashi:hcs98} is stated in a chapter where a global assumption
is made that $\OM$ is relatively compact. But it is easily checked that this condition is \textbf{not}
required in the proof of the latter lemma. Secondly, in \cite[Lemma~3.3.5]{kobayashi:hcs98},
the sets $V$ and $W$ appear as $X$-open neighbourhoods of some point in $\overline{\OM}$. For the ways
in which we will use Lemma~\ref{L:hyp}, it is more convenient to let $V$ and $W$ be as above. With
these words of explanation, we refer the reader to \cite{kobayashi:hcs98} for the proof of Lemma~\ref{L:hyp}.

\smallskip

We will need another technical lemma for the proof of Theorem~\ref{T:k-complete}.

\begin{lemma}\label{L:hol}
Let $X$ be a connected complex manifold and $G$ be an open set in $X$. Assume that there is a sequence 
$(z_n)_{n\geq 1}\subset G$ admitting the constants $R>0$ and $\nu\in\N$ such that $K_G(z_n,z_\nu) <R$ for all $n\geq\nu$. 
Then, there does not exist any holomorphic function $f\in \hol(G,\D)$ such that $f(z_n)\to 1$ as $n \to \infty$.
\end{lemma}

\begin{proof}
Let ${(z_n)_{n\geq 1}}$, $\nu$, $R$ be as given by the statement of the lemma. If possible let there exist a
holomorphic function $f\in\hol(G, \D)$ such that $f(z_n)\to 1$ as $n \to\infty$. Fix $\psi\in Aut(\D)$ 
such that $\psi(f(z_\nu))=0$. Since we know $\psi(S^1)=S^1$, so 
$\psi(1)=\lambda\in S^1$, say. Also $\psi\circ f\in\hol(G,\D)$. Therefore, for all $n\geq\nu$
\begin{align}
K_G(z_n, z_\nu) <R  \notag \\
\implies K_{\D}(\psi\circ f(z_n), \psi\circ f(z_\nu))\leq K_G(z_n, z_\nu) <R \notag \\
\implies \poin(0, {\lambda}^{-1}\psi\circ f(z_n))<R,\label{E:kine}
\end{align}
where $\poin$ is the Poincar{\'e} distance. Now, by our assumption about $f$, 
\begin{align*}
\lambda^{-1}\psi\circ f(z_n) &\lrarw 1 \\
\implies\poin(0, \lambda^{-1}\psi\circ f(z_n))&\lrarw\infty
\end{align*}
as $n\to \infty$. The last statement contradicts \eqref{E:kine}. Hence, the result.
\end{proof}

Now, we will state a definition and a proposition, which will be needed to prove our first theorem. The
following definition is due to Alehyane--Amal \cite{AlehyaneAmal:pdashdlefh03}.

\begin{definition}\label{D:weak-hyp}
Let $X$ be a complex manifold. We say that \emph{$X$ is weakly hyperbolic} if for every compact subset $K\subseteq X$, 
there exists a neighbourhood $U$ of $K$ in $X$ that satisfies the following two properties:
\begin{itemize}
\item[$(a)$] $U$ is Kobayashi hyperbolic;
\item[$(b)$] for all $f\in\smoo(\overline{\D}, X)$ such that $f$ is holomorphic on $\D$, if $f(\bdy \D)\subseteq K$ then $f(\overline{\D})\subseteq U$.
\end{itemize}
\end{definition}

\begin{proposition}\label{P:rel_com-imb}
Let $X$ be a complex manifold that is weakly hyperbolic. Let $\OM\varsubsetneq X$ be relatively compact 
domain in $X$. Then, $\OM$ is a hyperbolically imbedded domain in $X$.
\end{proposition}

\begin{proof}
Since $\OM$ is relatively compact in $X$, $\overline{\OM}$ is compact in $X$. Also, $X$ is weakly hyperbolic. 
Therefore, there exists an open set $U$ in $X$ such that $\overline{\OM}\subseteq U$ and $U$ is Kobayashi 
hyperbolic. Let $x\in\overline{\OM}$ and $V_x$ be an arbitrary neighbourhood of $x$ in $X$. Clearly $x\in V_x\cap U$ and $V_x\cap U$ is a
complex submanifold of $X$. Thus, there exists a
neighbourhood $W_x$ of $x$ in $U$ so that $\overline{W_x}\varsubsetneq V_x$,
and $\overline{W_{x}}$ is compact (and contained in $U$). Since $\overline{\overline{W_x}\cap\OM}$ is compact,
$\overline{\overline{W_{x}}\cap\OM}\cap\overline{\OM\setminus V_x}=\emptyset$, and $U$ is Kobayashi hyperbolic, we have
\begin{align*}
  K_\OM(\overline{W_x}\cap\OM, \OM\setminus V_x) &\geq K_U(\overline{W_x}\cap\OM, \OM\setminus V_x)\\
  &\geq K_U(\overline{\overline{W_x}\cap\OM}, \overline{\OM\setminus V_x}) >0,
\end{align*}
where the first inequality above is due to the fact that the inclusion $\OM\xhookrightarrow{} U$ is holomorphic.
Hence, the result. 
\end{proof}

Note that we did not use the full strength of Definition~\ref{D:weak-hyp} in the above proof. We will
comment more on this in Remark~\ref{Rem:weak_hyp}.
\medskip

\section{The proof of Theorem~\ref{T:weak_hyperbolic}}

We begin with a result that is crucial to the proof of Theorem~\ref{T:weak_hyperbolic}.

\begin{result}\label{R:cond-weak_hyp}
Let $X$ be any of the following: 
\begin{itemize}
\item[$(a)$] A Kobayashi hyperbolic complex manifold,
\item[$(b)$] A Stein manifold,
\item[$(c)$] A holomorphic fiber bundle $\pi: X\to Y$ such that
$Y$ is either of $(a)$ or $(b)$ and the fibers are Kobayashi hyperbolic,
\item[$(d)$] A holomorphic covering space $\pi: X\to Y$ such that $Y$ is either of $(a)$ or $(b)$.
\end{itemize}
Then, $X$ is weakly hyperbolic.
\end{result}

That $X$ as in $(a)$ is weakly hyperbolic is immediate from the definition (take $U$ in Definition~\ref{D:weak-hyp} to be $X$ itself).
When $X$ is as in $(b)$ or $(c)$, the above conclusion is proved in 
\cite{Duc:whscetwhs03}: see Example~2.3~(6) and Theorem~1.1 respectively discussed on \cite[p.~1017, 1015]
{Duc:whscetwhs03}.
When $X$ is as in $(d)$, the above conclusion is established in 
\cite{AlehyaneAmal:pdashdlefh03}: see (3) among the examples discussed on \cite[p.~203--204]
{AlehyaneAmal:pdashdlefh03}.

\begin{proof}[The proof of Theorem~\ref{T:weak_hyperbolic}]
Let $X$ be any of the $(a)$--$(d)$ above. Without loss of generality, we can assume that $X$ is connected. From Result~\ref{R:cond-weak_hyp}, it
follows that $X$ is weakly 
hyperbolic. Now, since $\OM\varsubsetneq X$ is a relatively compact domain in $X$, from
Proposition~\ref{P:rel_com-imb}, $\OM$ is a hyperbolically imbedded domain in $X$. Therefore, from part $(a)$ of 
Theorem~\ref{T:k-complete}, it follows that $(\OM, K_\OM)$ is Cauchy-complete.
\end{proof}

\begin{remark}\label{Rem:weak_hyp}
We must elaborate on the last observation in Section~\ref{S:prelim}: namely, that Proposition~\ref{P:rel_com-imb} does not require the
full strength of Definition~\ref{D:weak-hyp}. It turns out that the condition that is \emph{difficult} to ensure is $(a)$ of
Definition~\ref{D:weak-hyp}. The geometric or function-theoretic properties that ensure $(a)$ give us condition~$(b)$ of
Definition~\ref{D:weak-hyp} with great ease. Moreover, barring some small classes that are technical in their construction,
it does not seem easy to produce classes of complex manifolds that satisfy condition~$(a)$ of Definition~\ref{D:weak-hyp} \textbf{only}
and are not already considered by Result~\ref{R:cond-weak_hyp}.
For this reason, and because we do not wish to reproduce arguments already in \cite{AlehyaneAmal:pdashdlefh03, Duc:whscetwhs03}, we have
chosen to work with Definition~\ref{D:weak-hyp} (but treating it as an intermediary concept).
\end{remark}

\section{The proof of Theorem~\ref{T:k-complete}}\label{S:k-complete}

While Theorem~\ref{T:k-complete} resembles \cite[Theorem~1]{Gaussier:tachodic99} by Gaussier, the strategies
for their proofs are completely different. The proof of the latter result uses a localization principle
for the Kobayashi \emph{metric} which, in turn, relies on the mean-value inequality for certain plurisubharmonic
functions. With this strategy, one does not require $\OM\varsubsetneq \Cn$ to \emph{a priori} be Kobayashi
hyperbolic. This strategy applied to the setting of Theorem~\ref{T:k-complete} encounters an obstacle when it comes
to the existence of local anti-peak plurisubharmonic functions (which we shall not define here). The notion of
$\OM$ being hyperbolically imbedded provides the ``right'' condition for handling this obstacle\,---\,a condition that
is also rather nice, as Theorem~\ref{T:weak_hyperbolic} shows (cf. Result~\ref{R:cond-weak_hyp}). With the latter
hypothesis, we can provide a purely metric-geometry proof of Theorem~\ref{T:k-complete} where Lemma~\ref{L:hyp}
is the key technical tool.
\smallskip

Before we present the proof of Theorem~\ref{T:k-complete}, we should clarify that, for $X$ and $\OM$ as in this 
theorem, $\Top(X)$ will denote the manifold topology on $X$. With $\Top(\OM)$ having the analogous meaning, 
$\Top(\OM)$ is the relative topology on $\OM$ inherited from $\Top(X)$.

\begin{proof}[The proof of Theorem~\ref{T:k-complete}]
Since $\OM$ is a hyperbolically imbedded domain in $X$, $\OM$ is a Kobayashi hyperbolic domain in 
 $X$. Hence, $(\OM,K_\OM)$ is a metric space and
 \begin{align}\label{E:same_topo}
 \Top(\OM)=\Top(K_\OM),    
 \end{align}
where $\Top(K_\OM)$ denotes the metric topology on $(\OM,
 K_\OM)$. Let $(z_n)_{n\geq 1}$ be a Cauchy sequence in $(\OM,K_\OM)$. We claim that $(z_n)_{n\geq 1}$ converges 
 to some point in $\OM$ with respect to $\Top(K_\OM)$. The argument falls into two cases.
 
\medskip

\noindent{{\textbf{Case 1.}} \emph{The set $\{z_n: n\geq 1\}$ is relatively compact in $\Top(X)$.}}
\smallskip

\noindent{In this case, there exists some subsequence $(z_{n_k})_{k\geq 1}$ and $\wt{z}\in X$ such that
\begin{align}\label{E:sub_seq-lim}
 z_{n_k} & \lrarw \wt{z}\;\;\text{in} \;\; \Top(X)   
\end{align}
as $k\to\infty$. Hence, $\wt{z}\in \overline{\OM}\subseteq X$. Recall that $\overline{\OM}=\OM\cup\bdy{\OM}$. We first consider 
the subcase when $\wt{z}\in\OM$. Then, from \eqref{E:sub_seq-lim} and \eqref{E:same_topo}, it follows that
\begin{align*}
z_{n_k} &\lrarw \wt{z} \;\; \text{in} \;\; \Top(\OM)\\
\implies z_{n_k} &\lrarw \wt{z} \;\; \text{in} \;\; \Top(K_\OM)
\end{align*}
as $k\to\infty$. Since $(z_n)_{n\geq 1}$ is a Cauchy sequence in $(\OM,K_\OM)$, from above it follows that $z_n\to\wt{z}\in\OM$ in $\Top(K_\OM)$. 
Hence, our claim is true for this subcase.
\smallskip

We now consider the subcase when $\wt{z}\notin\OM$. Then, clearly $\wt{z}\in\bdy{\OM}$. Without loss of generality, we may assume
\begin{align*}
z_n\lrarw\wt{z}\in\bdy\OM\;\;\text{in}\;\;\Top(X)
\end{align*}
 as $n\to\infty$. From our assumption there exists a local holomorphic weak peak function at $\wt{z}$. Hence, there exists a 
neighbourhood $V_{\wt{z}}$ of $\wt{z}$ in $X$ and a continuous function $f_{\wt{z}}: (V_{\wt{z}}\cap\OM)\cup\{\wt{z}\}\lrarw \C$ 
such that $f_{\wt{z}}$ is holomorphic on $V_{\wt{z}}\cap\OM$ and such that 
\begin{align}\label{E:peak_function-case-1}
|f_{\wt{z}}(z)|<1 \;\;\forall z\in V_{\wt{z}}\cap \OM, \quad f_{\wt{z}}(\wt{z})=1.
\end{align}
Now, since $\OM$ is hyperbolically imbedded domain in $X$ and $V_{\wt{z}}$ is a neighbourhood of $\wt{z}$ in $X$, there exists a neighbourhood $W_{\wt{z}}$ of $\wt{z}$ in $X$ such that $\overline{W_{\wt{z}}}\subseteq 
V_{\wt{z}}$ and
\[\delta_{\wt{z}}:= K_\OM(W_{\wt{z}}\cap\OM, \OM\setminus V_{\wt{z}})>0.\]
Recall that $(z_n)_{n\geq 1}$ is a Cauchy sequence in $(\OM,K_\OM)$ such that $z_n\lrarw\wt{z}$ in $\Top(X)$ as $n\to\infty$. Hence, there exists 
$N\in \N$ such that for all $m,n\geq N$
\begin{align*}
K_\OM(z_m,z_n)<\delta_{\wt{z}}/4\quad \text{and}\quad z_n\in W_{\wt{z}}\cap\OM.  
\end{align*}
Therefore, from Lemma~\ref{L:hyp} (take $W=W_{\wt{z}}$, $V=V_{\wt{z}}$, and $\varepsilon=\delta_{\wt{z}}/4\in(0,\delta_{\wt{z}}/2)$),
there exists a constant $\wt{C}\equiv \wt{C}(\delta_{\wt{z}})>0$ such that for all $m,n\geq N$
\begin{align}
K_{V_{\wt{z}}\cap\OM}(z_m,z_n)<\wt{C}\delta_{\wt{z}}/4 \notag \\
\implies K_{V_{\wt{z}}\cap\OM}(z_n,z_N)<\wt{C}\delta_{\wt{z}}/4. \label{E:K-ineq_thm-2}
\end{align}
Furthermore, since $f_{\wt{z}}\in\hol(V_{\wt{z}}\cap\OM,\D)$ and
$f_{\wt{z}}(z_n)\to 1$
as $n\to\infty$, \eqref{E:K-ineq_thm-2} contradicts Lemma~\ref{L:hol} (take the sequence $(z_n)_{n\geq N}$, $G=V_{\wt{z}}\cap\OM$ and 
$R=\wt{C}\delta_{\wt{z}}/4<\infty$). Hence, our claim is true for Case~1.
}

\medskip

Now if $\OM$ is a relatively compact domain in $X$, then Case~1 is the only possibility for any sequence $(z_n)_{n\geq 1}$ in
$\OM$. Hence, the proof of part $(a)$ is done.

\smallskip

Now, we assume $\OM$ is not relatively compact in $X$ and consider the remaining case:

\medskip

\noindent{{\textbf{Case 2.}} \emph{The set $\{z_n:n\geq 1\}$ is not relatively compact in $\Top(X)$.}}
\smallskip

\noindent{In what follows, if $S$ is a subset of $X$, $\overline{S}$ will denote the closure relative to
$\Top(X)$. In this case, $X$ is non-compact. We consider the one-point compactification of $X$, $X^\infty$
which is compact. Hence, there exists a 
subsequence $(z_{n_k})_{k\geq 1}$ such that
$z_{n_k}\to\infty$ in $\Top(X^\infty)$ as $k\to\infty$. Relabelling this subsequence, we may, without loss
of generality, say $z_n\to\infty$ in $\Top(X^\infty)$ as $n\to\infty$. Also, by assumption
$\infty$ admits a local holomorphic weak peak function. Hence, there exists a compact set $K\subset\OM$ and a function
$f_\infty$ holomorphic on $\OM\setminus K$ such that 
\begin{align}\label{E:peak-function-case_2}
 |f_\infty (z)| <1\;\; \forall z\in\OM\setminus K,\quad \lim_{z\to\infty} f_\infty(z)=1.
\end{align}
Since $K$ is compact in $\OM$ and $\OM$ is locally compact, there exists an open set $U$ in $\OM$ such that $K\varsubsetneq 
U\subseteq\overline{U}\varsubsetneq\OM$ and $\overline{U}$ is compact. Let $V:=X\setminus K$ and $W:=X\setminus \overline{U}$, which are open sets in $X$.
Now, $K\varsubsetneq U\subseteq\overline{U}\varsubsetneq\OM$ implies
\begin{align}\label{E:compact-closed-disjoint}
\OM\setminus\overline{U}\subseteq \OM\setminus K, \quad \bdy U&\subseteq
\OM\setminus K \notag \\
\implies \overline{\OM\setminus \overline{U}}\cap\OM &\subseteq\OM\setminus K \notag\\
\implies (\overline{W\cap\OM}\cap \OM)\cap K&=\emptyset.
\end{align}
Let 
\[
\delta_\infty := K_\OM(W\cap\OM, \OM\setminus V)
\geq K_\OM((\overline{W\cap\OM}\cap\OM), K)>0,
\]
where the last inequality follows from \eqref{E:compact-closed-disjoint} 
and the fact that $K$ is compact
and $(\overline{W\cap\OM}\cap\OM)$ is closed with respect to $\Top(K_\OM)$.
Recall that (after relabelling) $(z_n)_{n\geq 1}$ is a Cauchy sequence in $(\OM, K_\OM)$ such that $z_n\to\infty$ in $\Top(X^\infty)$ as $n\to\infty$. Hence, there exists 
$N\in \N$ such that for all $m,n\geq N$
\begin{align*}
K_\OM(z_m,z_n)<\delta_{\infty}/4 \quad \text{and} \quad z_n\in W\cap\OM.  
\end{align*}
Therefore, from Lemma~\ref{L:hyp} (take $\varepsilon=\delta_\infty/4\in(0,\delta_\infty/2$)),
there exists a constant $C\equiv C(\delta_\infty)>0$ such that for all $m,n\geq N$,
\begin{align}
K_{\OM\cap V}(z_m,z_n)<C\delta_\infty/4 \notag \\
\implies K_{\OM\setminus K}(z_n,z_N)<C\delta_\infty/4. \label{E:K-infty-ineq_thm-2}
\end{align}
Furthermore, since $f_\infty\in\hol(\OM\setminus K,\D)$ and $f_\infty(z_n)\to 1$
as $n\to\infty$, \eqref{E:K-infty-ineq_thm-2} contradicts Lemma~\ref{L:hol} (take the sequence $(z_n)_{n\geq N}$, $G=\OM\setminus K$ and 
$R=C\delta_\infty/4<\infty$). Hence, Case~2 cannot occur.
}
\smallskip

These two cases establish the result.
\end{proof}
\smallskip

\section*{Acknowledgements}
I would like to thank my thesis advisor, Prof. Gautam Bharali, for many useful discussions during course of this work.
I am also grateful for his advice on the writing of this paper. This work is supported in part by a scholarship from the Prime
Minister's Research Fellowship (PMRF) programme (fellowship no.~0201077) and by a DST-FIST grant (grant no.~TPN-700661).


\smallskip



\begin{thebibliography}{88}

\bibitem{AlehyaneAmal:pdashdlefh03}
Omar Alehyane and Hichame Amal, \emph{Propri{\'e}t{\'e}s d'extension et applications s{\'e}par{\'e}ment holomorphes dans les espaces faiblement hyperboliques}, Ann. Polon. Math. {\bf 81} (2003), no.~3, 201--215.

\bibitem{BedfordFornaess:acopfowpd78}
Eric Bedford and John Erik Fornaess, \emph{A construction of peak functions on weakly pseudoconvex domains}, Ann. of Math. {\bf 107} (1978), no.~3, 555--568.

\bibitem{BharaliEtAl:phmspRs18} 
Gautam Bharali, Indranil Biswas, Divakaran Divakaran, Jaikrishnan Janardhanan, \emph{Proper holomorphic mappings onto
symmetric products of a Riemann surface}, Doc. Math. {\bf 23} (2018), 1291--1311.

\bibitem{Duc:whscetwhs03}
Pham Viet Duc, \emph{On weakly hyperbolic spaces and a convergence-extension theorem in weakly hyperbolic spaces}, Internat. J. Math. {\bf 14} (2003), no.~10, 1015--1024.

\bibitem{Gaussier:tachodic99}
Herv{\'e} Gaussier, \emph{Tautness and complete hyperbolicity of domains in $\Cn$}, Proc. Amer. Math. Soc. {\bf 127} (1999), no.~1, 105--116.

\bibitem{JarnickiPflug:idmca13}
Marek Jarnicki and Peter Pflug, Invariant Distances and Metrics in Complex Analysis,
Walter de Gruyter \& Co., 1993.

\bibitem{Kaup:hkR68}
Wilhelm Kaup, \emph{Hyperbolische komplexe R{\"a}ume}, Ann. Inst. Fourier (Grenoble) {\bf 18}, fasc.~2 (1968), 303–-330.

\bibitem{KimVerdiani:cndnrnn2dag04}
Kang-Tae Kim and Luigi Verdiani, \emph{Complex $n$-dimensional manifolds with a real $n^2$-dimensional automorphism group},
J. Geom. Anal. {\bf 14} (2004), no.~4, 701--713.

\bibitem{kobayashi:hmhm70}
Shoshichi Kobayashi, Hyperbolic Manifolds and Holomorphic Mappings, Pure and Appl. Math. {\bf 2},
Marcel Dekker, Inc., New York, 1970.

\bibitem{kobayashi:hcs98}
Shoshichi Kobayashi, Hyperbolic Complex Spaces, Grundlehren Math. Wiss. {\bf 318}, Springer-Verlag, Berlin, 1998.

\bibitem{Kwack:gbPt69}
Myung H. Kwack, \emph{Generalization of the big Picard theorem}, Ann. of Math. {\bf 90} (1969), 9--22.

\bibitem{Zwonek:ftpspcm20}
W{\l}odzimierz Zwonek, {\em Function theoretic properties of symmetric powers of complex manifolds},
J. Geom. Anal. {\bf 30} (2020), no.~2, 1226--1237.

\end{thebibliography}
\end{document}